%% file: structure-of-elasticity.tex
\newif\ifArxiv\Arxivtrue
\ifArxiv
\documentclass[a4paper,leqno]{amsart}
\usepackage[utf8]{inputenc}
\usepackage{times,euler}
\usepackage{amssymb}
\usepackage{a4wide}
\usepackage{hyperref}
\let\Cite\cite
\let\Cites\cite
\let\cfcite\cite
\else
\fi
\usepackage[T1]{fontenc}
\usepackage[capitalize]{cleveref}
\usepackage{enumitem}\setlist[enumerate]{label=(\roman*)}
\usepackage{amssymb}
\usepackage{tikz-cd}
\usepackage{booktabs}
\usepackage{ao-math-symbols}
\usepackage{ao-math-fields}
\usepackage{ao-math-opt}
\usepackage{luh-colors}

\ifArxiv
\input{amsorcid.add}
\theoremstyle{plain}
\newtheorem{lemma}{Lemma}[section]
\newtheorem{theorem}[lemma]{Theorem}
\newtheorem{proposition}[lemma]{Proposition}
\theoremstyle{definition}
\newtheorem{definition}[lemma]{Definition}
\newtheorem{problem}[lemma]{Problem}
\theoremstyle{remark}
\newtheorem{remark}[lemma]{Remark}
\numberwithin{equation}{section}
\else
\fi

\let\bd\boundary
\let\pd\partial
\let\dfn\define
\newcommand\bdD{\Gamma_D}
\newcommand\bdN{\Gamma_N}
\newcommand\extD{\mathcal F}
\newcommand\extN{\mathcal L}
\newcommand\Rd{\R^d}
\newcommand\Md{\mathbb M^d}
\newcommand\Sd{\mathbb S^d}
\newcommand\Ad{\mathbb A^d}
\newcommand\setD{\mathcal D}
\newcommand\data{\mathfrak D}
\newcommand\ed{\~e}
\newcommand\sd{\~s}
\newcommand\on{\text{ on }}
\newcommand\qfor{\quad\forall}
\newcommand\tfor{\ \forall}
\newcommand\cg{\mathfrak g}
\newcommand\cG{\mathfrak G}
\newcommand\lM{\ell_s}
\newcommand\lN{\ell_e}
\newcommand\dy{\,\textup dy}
\newcommand\id[1][]{\textup{id}_{#1}}
\newcommand\bprod{\sprod[\Gamma]}
\newcommand\iprod[3][]{(#2, #3)_{#1}}
\newcommand\cl[2][1.5mu]{\mkern#1
  \overline{\mkern-#1 #2\mkern-#1}\mkern#1}
\newcommand\LetO[1][]{Let $\Omega \subset \Rd$ be a bounded Lipschitz domain#1. }
\newcommand\LetOG{\LetO[, $\Gamma = \bd\Omega$]}
\newcommand\LetODN{Let $\Omega, \bdD, \bdN$ be as introduced above. }

\newcommand\Ltwo{L^2(\Omega)}
\newcommand\LtwoM{L^2(\Omega, \Md)}
\newcommand\LtwoR{L^2(\Omega, \Rd)}
\newcommand\LtwoS{L^2(\Omega, \Sd)}
\newcommand\Hone[1][\Rd]{H^1(\Omega, #1)}
\newcommand\HoneD{H^1_D(\Omega, \Rd)}
\newcommand\Honez{H^1_0(\Omega, \Rd)}
\newcommand\Hdiv[1][\Sd]{H^{\div}(\Omega, #1)}
\newcommand\HdivN{H^{\div}_N(\Omega, \Sd)}
\newcommand\Hdivz[1][\Sd]{H^{\div}_0(\Omega, #1)}
\newcommand\Hsym[1][]{H_{#1}\tsp{sym}(\Omega, \Rd)}
\newcommand\Hphalf[1][]{H^{1/2}(\Gamma_{#1}, \Rd)}
\newcommand\Hnhalf[1][]{H^{-1/2}_s(\Gamma_{#1}, \Rd)}
\newcommand\Hnhalff{H^{-1/2}(\Gamma, \Rd)}
\newcommand\test[1][\Rd]{\mathscr D(\Omega, #1)}
\newcommand\dist{\mathscr D'(\Omega, \Rd)}
\newcommand\obj{\varphi}
\newcommand\BPD[1]{BPD$_{#1}$}
\newcommand\BPM[1]{BPM$_{#1}$}
\newcommand\BPN[1]{BPN$_{#1}$}
\newcommand\XD[1]{\XM{D#1}}
\newcommand\XM[1]{X_{#1}}
\newcommand\XN[1]{\XM{N#1}}

\let\div\relax
\DeclareMathOperator*{\argmin}{argmin}
\DeclareMathOperator{\div}{div}
\DeclareMathOperator{\grad}{grad}
\DeclareMathOperator{\Grad}{grad^s}
\DeclareMathOperator{\graph}{graph}
\DeclareMathOperator{\im}{im}
\DeclareMathOperator{\dom}{dom}
\DeclareMathOperator{\Span}{span}
\DeclareMathOperator{\tr}{tr}

\newcommand\tiText{Intrinsic Structure of Elasticity in Hilbert Space}
\newcommand\abText{%
  Mathematical elasticity has a long history and a huge body of literature.
  Surprisingly, at its heart elasticity exhibits
  a rich and transparent structure that, to our knowledge,
  is rarely presented in an explicit and unified form.
  Using standard tools from functional analysis,
  this article reveals four orthogonal decompositions of Hilbert spaces
  that characterize the static equilibrium
  of an elastic body at small deformations,
  universally for arbitrary spatial dimension, arbitrary boundary conditions,
  and classical as well as data-driven formulations.
  Regarding data-driven continuum mechanics,
  it highlights the intrinsic structure that remains unchanged
  when constitutive laws are replaced by material data sets.
}
\newcommand\kwText{%
  Elasticity\ksep
  PDE boundary value problems\ksep
  data-driven formulations\ksep
  structural analysis\ksep
  duality\ksep
  orthogonal decompositions%
}
\newcommand\clText{%
  35Q74, 
  35J57, 
  47N50, 
  49N10, 
  49N15
}

\newcommand\UiB{University of Bergen,
  Geophysical Institute and Bergen Offshore Wind Centre (BOW)\asep
  Allégaten~70, 5007 Bergen, Norway}
\newcommand\LUH{Leibniz University Hannover,
  Institute of Applied Mathematics\asep
  Welfengarten~1, 30167 Hannover, Germany}

\ifArxiv
\newcommand\asep{, }
\newcommand\ksep{, }
\title{\tiText}
\author[C. G. Gebhardt]{Cristian G. Gebhardt}
\address{Cristian G. Gebhardt\\\UiB}
\orcid{0000-0003-0942-5526}
\email{cristian.gebhardt@uib.no}
\urladdr{uib.no/en/persons/Cristian.Guillermo.Gebhardt}

\author[J. Lankeit]{Johannes Lankeit}
\address{Johannes Lankeit\\\LUH}
\orcid{0000-0002-2563-7759}
\email{lankeit@ifam.uni-hannover.de}
\urladdr{ifam.uni-hannover.de/lankeit}

\author[M. C. Steinbach]{Marc C. Steinbach}
\address{Marc C. Steinbach\\\LUH}
\orcid{0000-0002-6343-9809}
\email{mcs@ifam.uni-hannover.de}
\urladdr{ifam.uni-hannover.de/steinbach}

\begin{document}

\begin{abstract}
  \abText
\end{abstract}


\subjclass[2000]{\clText}

\date{\today}

\maketitle
\else
\fi

\section{Introduction}

This work is motivated by static equilibrium problems
in data-driven elasticity at small deformations,
which generalize PDE boundary value problems
that involve the symmetric gradient operator for vector fields
and the divergence operator for symmetric tensor fields
\cite{Ciarlet:1988,MarsdenHughes1994}.
Many other classical problems in physics and engineering---%
such as Darcy flow \cite{Bear1972},
steady‑state heat conduction \cite{CarslawJaeger1959},
diffusion \cite{Crank1975},
electrostatics and magnetostatics \cite{Jackson1998}---%
share a common mathematical foundation:
they are all governed by linear, second-order elliptic PDEs in divergence form
\cite{GilbargTrudinger2001,Evans2010}.
Each model consists of a constitutive law
that links a flux to a gradient of a potential,
and a balance law expressed through a divergence.
In scalar problems (Darcy, heat conduction, diffusion, electrostatics),
the flux takes the form $q = -A \grad u$,
where $A$ is a symmetric positive-definite tensor,
and the governing equation becomes
\begin{equation*}
  \div(A \grad u) = f,
\end{equation*}
where $f$ is a source term.
In elasticity at small deformations, the structure is similar,
but \ifArxiv the gradient \else $\grad u$ \fi
is replaced by the symmetric gradient
$\Grad u = \frac12 (\grad u + \grad u\tp)$,
reflecting that stress depends only on deformation,
not rigid rotations or translations.
The constitutive relation $s = C : \Grad u$
and the equilibrium condition $\div s + f = 0$
together yield the vector‑valued elliptic system
\begin{equation*}
  \div (C : \Grad u) = -f.
\end{equation*}
Thus, despite differences in physical meaning and interpretation,
all these models are built on the same duality:
a gradient‑type operator that creates a generalized flux
(heat flux, Darcy velocity, electric displacement, elastic stress),
and a divergence operator that enforces conservation
(of mass, energy, charge, or momentum).
This shared operator structure explains
why these problems have similar analytical properties,
variational formulations, and numerical discretizations,
and why techniques developed in one field transfer naturally to another.
The mathematical structure also extends naturally to data‑driven elasticity,
where the constitutive law is no longer given analytically
but represented through a data set of admissible strain-stress pairs.
Even in this setting, the PDE governing the displacement field
retains its fundamental form:
the symmetric gradient maps displacements to strains,
and the divergence maps stresses to forces,
preserving the duality between kinematics and equilibrium.
However, much of the required mathematical theory
has historically been developed within specific application domains,
primarily in physics and engineering.
Although most relevant aspects exist somewhere in the extensive literature,
they are often difficult to locate, may lack mathematical rigor,
or require substantial generalization.

Our article addresses these issues
for the problem class of interest
by providing a concise and rigorous theoretical framework
that covers arbitrary boundary conditions (Dirichlet, mixed, Neumann)
in any spatial dimension,
that generates insight into the inherent common structure
and into subtle differences of the three cases
along with their physical interpretations,
and that is presented on a basic technical level
using few standard concepts.
In specific, we consider the divergence and gradient
as unbounded operators on Hilbert spaces
and as bounded operators on their domains,
we use trace operators and associated extension and lift maps,
and we need infinitesimal rigid body motions
and the two key inequalities of Poincaré and Korn.
To uncover the intrinsic problem structure,
we construct four orthogonal decompositions
of Hilbert spaces by means of adjoint operators and pseudoinverses.
Moreover, we introduce data-driven problem formulations
that add an optimization context with the associated duality.
Analyzing these problems provides further structural insight
and leads to a unified formulation
that covers all problems that we address.

Since even a moderately complete literature review
is entirely out of reach
(and probably more distracting than helpful),
we only cite a few standard works,
selected rigorous sources for specific results,
and selected references for further reading.
Just recently we came across the closely related article
\Cite{Gudoshnikow_Krizk:2025} wherein adjoint duality
is used to analyze inhomogeneous linear elasticity
under Dirichlet or mixed boundary conditions
and where the technical details from functional analysis
are treated extensively, with numerous rigorous references.
However, the coupling of PDE and boundary conditions in defining
divergence and gradient as unbounded adjoints
obscures certain symmetries and other structural aspects.
Details will be discussed in the summary.

The remainder of this article is structured as follows.
In \cref{sec:preliminaries} we introduce basic notation and concepts,
and we derive fundamental results for our setting,
in particular the four orthogonal decompositions of Hilbert spaces.
\Cref{sec:dd-problems} presents the data-driven paradigm and analyzes
successively more concise formulations of the static elasticity problem.
Using duality theory from optimization,
we finally arrive at the unified problem formulation:
an $L^2$ minimization problem for projected auxiliary strains and stresses
with values in the material data set.
The results are illustrated in \cref{sec:1d}, where spaces, operators,
and certain solutions are explicitly given for spatial dimension one.
We conclude with a summary in \cref{sec:summary}.

\section{Preliminaries}
\label{sec:preliminaries}

We denote the space of square matrices by $\Md \dfn \R^{d \times d}$
and the subspaces of antisymmetric (skew-symmetric)
and symmetric matrices by $\Ad \dfn \smash[b]{\R^{d \times d}\tsb{skew}}$
and $\Sd \dfn \smash[b]{\R^{d \times d}\tsb{sym}}$, respectively.
It holds $\Md = \Ad \oplus \Sd$ where
$\Ad \perp \Sd$ with respect to the Frobenius scalar product,
$\iprod[\Md]{A}{B} \dfn \tr(A\tp B)$.
We address problems described by vector fields in $\Rd$
and symmetric tensor fields in $\Sd$.

Throughout, let $\Omega \subset \Rd$ be a bounded domain
(i.e., an open and connected set)
with Lipschitz boundary $\Gamma \dfn \bd\Omega$,
and let $\Gamma$ be partitioned into
a Dirichlet boundary $\bdD$ and a Neumann boundary $\bdN$,
both open subsets relative to~$\Gamma$,
\begin{align*}
  \Gamma &= \cl\Gamma_D \cup \cl\Gamma_N,
  & \bdD \cap \bdN &= \0.
\end{align*}
Subsequently, we will work with \ifArxiv the \fi Hilbert spaces
$\LtwoR$, $\LtwoS \subset \LtwoM$, $\Hone$ and
$\Hdiv \subset \Hdiv[\Md] = \defset{s \in \LtwoM}{\div s \in \LtwoR}$,
each equipped with the standard scalar product. In particular, we have
\begin{gather*}
  \iprod[\Hone]{u}{v} =
  \iprod[\LtwoR]{u}{v} + \iprod[\LtwoM]{\grad u}{\grad v}, \\
  \iprod[\Hdiv]{s}{t} \equiv \iprod[{\Hdiv[\Md]}]{s}{t} =
  \iprod[\LtwoM]{s}{t} + \iprod[\LtwoR]{\div s}{\div t},
\end{gather*}
where\ifArxiv in\fi\ the $L^2$ integrands
are finite-dimensional scalar products
such as $\iprod[\Rd]{u(x)}{v(x)}$ and $\iprod[\Md]{s(x)}{t(x)}$.
As usual, the operators $\div$, $\grad$ and $\Grad$ are defined row-wise:
$(\div s)_i = \sum_j \pd_j s_{ij}$,
$(\grad u)_{ij} = \pd_j u_i$ and
$(\Grad u)_{ij} = \frac12 (\pd_i u_j + \pd_j u_i)$.
To handle the boundary conditions, we need
$\Honez$, the closure of $\test$ in $\Hone$, and
$\Hdivz[\Md]$, the closure of $\test[\Md]$ in $\Hdiv[\Md]$,
where $\test[X] = C_c^\oo(\Omega, X)$ is the space of
$X$-valued test functions. Moreover,
we will rely heavily on the following standard trace operators,
and later on associated extension and lifting operators.

\begin{lemma}[Trace maps]
  \label{lem:traces}
  \LetOG
  Then:
  \begin{enumerate}
  \item There \ifArxiv exists \else is \fi a surjective bounded linear map
    $T_0\: \Hone \to \Hphalf$ such that
    $T_0 u = u|_\Gamma\ifArxiv$ for every
    $\else \tfor\fi u \in C^1(\cl\Omega, \Rd)$.
  \item There is a surjective bounded linear map
    $T_\nu\: \Hdiv[\Md] \to \Hnhalff$ such that
    $T_\nu s = s \nu|_\Gamma \ifArxiv$ for every
    $\else \tfor\fi s \in C^1(\cl\Omega, \Md)$,
    where $\nu$ is the outer unit normal.
  \item It holds $\ker T_0 = \Honez$ and $\ker T_\nu = \Hdivz[\Md]$.
  \end{enumerate}
\end{lemma}

\begin{proof}
  For the scalar-vector case, $u \in \Hone[\R]$ and $s \in \Hdiv[\Rd]$,
  the claims are proved in
  \cite[Thms.\ I.1.5, I.2.5, I.2.6, Cor.\ I.2.8]{Girault_Raviart:1986},
  see also \cite[Thms. 6.8.13, 6.9.2]{Kufner_John_Fucik:1977} for $T_0$
  and \Cite{Sohr:2001} for $T_\nu$.
  They generalize readily (row-wise)
  to $u \in \Hone$ and $s \in \Hdiv[\Md]$, respectively.
\end{proof}

We call $T_0$ the \emph{Dirichlet trace map}
and $T_\nu$ the \emph{normal trace map}.
The space $\Hphalf$ is actually defined as the range
$\im T_0 \subset L^2(\Gamma, \Rd)$ whereas
$\Hnhalff$ is defined as the dual space $\Hphalf^*$\ifArxiv\else,\fi\
\cfcite[Eq.\ (I.3.6.9)]{Sohr:2001}.
Let $\bprod{}{}$ denote the duality pairing.
To work with $\Sd \subset \Md$, let $\Hnhalf \subseteq \Hnhalff$
denote the (closed) range of $T_s \dfn T_\nu|_{\Hdiv}$.
Then, similar to \cref{lem:traces} (ii), (iii),
we have the \emph{surjective} bounded linear map
\begin{equation*}
  T_s\: \Hdiv \to \Hnhalf \qtextq{with} \ker T_s = \Hdivz.
\end{equation*}
We will write $T_\nu$ rather than $T_s$
whenever the restriction is irrelevant.
Throughout this work, we consider only those pairs
$(s, u) \in \Hdiv \times \Hone$ that satisfy
\begin{equation*}
  \bprod{T_\nu s}{T_0 u} = 0.
\end{equation*}
In particular, this orthogonality condition
holds for every pair $(s, u)$ in the product of the closed subspaces
\begin{equation*}
  \HdivN = \defset{s \in \Hdiv}{T_\nu s = 0 \on \bdN}
\end{equation*}
and
\begin{equation*}
  \HoneD = \defset{u \in \Hone}{T_0 u = 0 \on \bdD}.
\end{equation*}

\begin{remark}
  \label{rem:Neumann-meaning}
  By definition, the Neumann condition $T_\nu s = 0$ on $\bdN$ means that
  $\bprod{T_\nu s}{g} = 0$ for every $g \in \Hphalf$
  that obeys $g = 0$ on $\Gamma\setminus\bdN$,
  equivalently for every $g \in \im T_0|_{\HoneD}$ (by surjectivity of $T_0$).
\end{remark}

\begin{theorem}[Integration by parts]
  \label{thm:div-grad-integral}
  \LetOG
  For every $s \in \Hdiv$ and $u \in \Hone$ it then holds
  \begin{equation*}
    \iprod[\LtwoR]{\div s}{u} +
    \iprod[\LtwoS]{s}{\Grad u} =
    \bprod{T_\nu s}{T_0 u}
    .
  \end{equation*}
\end{theorem}

\begin{proof}
  For the scalar-vector case, $u \in \Hone[\R]$ and $s \in \Hdiv[\Rd]$,
  the result is proved in
  \cite[Thm.\ 1.5.3.1]{Grisvard:1985}, \cite[Thm. 3.1.1]{Necas:2012},
  \cite[Lem.\ II.1.2.3]{Sohr:2001},
  and it generalizes row-wise to
  $u \in \Hone$ and (not necessarily symmetric) $s \in \Hdiv[\Md]$,
  \begin{equation*}
    \iprod[\LtwoR]{\div s}{u} +
    \iprod[\LtwoM]{s}{\grad u} =
    \bprod{T_\nu s}{T_0 u}.
  \end{equation*}
  For symmetric $s \in \Hdiv$, the identity
  $\iprod[\LtwoM]{s}{\grad u} = \iprod[\LtwoS]{s}{\Grad u}$ gives the claim.
  (In \Cite{Conti2018} the result for $s \in \Hdiv[\Md]$
  is stated without proof.)%
  \ifArxiv\else\hbox{\quad}\fi
\end{proof}

\begin{lemma}[Moore--Penrose inverse]
  \label{lem:psi}
  \ifArxiv
  Let $X, Y$ be Hilbert spaces and $A \in L(X, Y)$ with closed range. \else
  Let $A \in L(X, Y)$ with closed range where $X, Y$ are Hilbert spaces. \fi
  Then, $A$ has a unique pseudoinverse $A^+ \in L(Y, X)$
  that satisfies the Penrose axioms (with orthogonal projections
  $P_{\im A}$ and $P_{\ker A}^\perp = P_{(\ker A)^\perp}$),
  \begin{align*}
    (1)\kern1.7em A A^+ A &= A, & (3)\quad (A A^+)^* &= A A^+ = P_{\im A}, \\
    (2)\quad A^+ A A^+ &= A^+, & (4)\quad (A^+ A)^* &= A^+ A = P_{\ker A}^\perp.
  \end{align*}
  If $A$ is injective with closed range,
  then $A^+$ is a left-inverse and hence surjective;
  if $A$ is surjective,
  then $A^+$ is a right-inverse and hence injective.
  Moreover, it holds
  \begin{align*}
    A^+ &= \Inv{(A^* A)} A^* \qtext{($A$ injective)},
    &A^+ &= A^* \Inv{(A A^*)} \qtext{($A$ surjective)}.
  \end{align*}
\end{lemma}

\begin{proof}
  Observe that $A|_{(\ker A)^\perp}\: \ker A^\perp \to \im A$ is invertible,
  define $A^+|_{\im A} \dfn \Inv{(A|_{(\ker A)^\perp})}$
  and $A^+|_{(\im A)^\perp} \dfn 0$, and verify the claims.
  (See also \Cites{Penrose:1955,Ben-Israel_Greville:2003}.)
\end{proof}

\begin{lemma}[Trace extension]
  \label{lem:ext-D}
  \LetODN
  Then, there \ifArxiv exists \else is \fi
  a unique bounded linear extension operator $\extD$
  that maps Dirichlet boundary data from $\bdD$
  to the orthogonal complement of $\HoneD$ in $\Hone\:$
  \begin{align*}
    \extD\: \Hphalf &\to \HoneD^\perp,
    &T_0 \extD g &= g \on \bdD \qfor g \in \Hphalf.
  \end{align*}
\end{lemma}

\begin{proof}
  Since $T_0$ is surjective with $\ker T_0 = \Honez$,
  $\extD_0 \dfn T_0^+$ is the unique right-inverse of $T_0$
  such that $\im \extD_0 = \Honez^\perp$.
  Then, the desired extension map is $\extD \dfn P^\perp \extD_0$
  where $P, P^\perp$ denote the orthogonal projections
  onto $\HoneD$ and its complement $\HoneD^\perp$ in $\Hone$.
  Indeed, $T_0 \extD g = g$ on $\bdD$ follows from
  \begin{equation*}
    T_0 \extD g
    =
    T_0 P^\perp \extD_0 g
    =
    T_0 (\id[\Hone] - P) \extD_0 g
    =
    g - T_0 P \extD_0 g
    .
    \qedhere
  \end{equation*}
\end{proof}

\begin{lemma}[Normal trace lifting]
  \label{lem:ext-N}
  \LetODN
  Then, there \ifArxiv exists \else is \fi
  a unique bounded linear lifting operator $\extN$
  that maps Neumann boundary data from $\bdN$
  to the orthogonal complement of $\HdivN$ in $\Hdiv\:$
  \begin{align*}
    \extN\: \Hnhalf &\to \HdivN^\perp,
    &T_\nu \extN h &= h \on \bdN \qfor h \in \Hnhalf.
  \end{align*}
\end{lemma}

\begin{proof}
  Since $\im T_s = \Hnhalf$ and $\ker T_s = \Hdivz$,
  $\extN_\nu \dfn T_s^+$ is the unique right-inverse of $T_s$
  such that $\im \extN_\nu = \Hdivz^\perp$.
  Then, the desired lifting map is $\extN \dfn P^\perp \extN_\nu$
  where $P, P^\perp$ denote the orthogonal projections
  onto $\HdivN$ and $\HdivN^\perp$ in $\Hdiv$.
  Indeed, $T_\nu \extN h = h$ on~$\bdN$ follows from
  \begin{equation*}
    T_\nu \extN h
    =
    T_\nu P^\perp \extN_\nu h
    =
    T_\nu (\id[\Hdiv] - P) \extN_\nu h
    =
    h - T_\nu P \extN_\nu h
    .
    \qedhere
  \end{equation*}
\end{proof}

\begin{remark}
  The Hilbert space structure of $\Hphalf$ and $\Hnhalf$ is not needed
  in \cref{lem:ext-D,lem:ext-N} (or elsewhere in this work)
  because $T_0$ and $T_s$ are surjective.
  Indeed, $A^+$ in \cref{lem:psi} still exists for surjective $A$
  if $Y$ is a Banach space and the decomposition
  $\im A \oplus (\im A)^\perp = Y \oplus \set{0}$
  is regarded to be ``orthogonal''.
\end{remark}

If $X$, $Y$ are Hilbert spaces, an \emph{unbounded linear operator}
$T\: \dom(T) \subseteq X \to Y$ is a (not necessarily bounded) linear map
whose domain $\dom(T)$ is a linear subspace of $X$,
with range $\im T \subseteq Y$. It is called \emph{closed} iff
$\graph(T) = \defset{(x, Tx)}{x \in \dom(T)}$
is a closed subspace of $X \times Y$.
If $T$ is \emph{densely defined}, i.e., $\dom(T) \subseteq X$ is dense,
we define the adjoint $T^*\: \dom(T^*) \subseteq Y \to X$ by letting
\begin{equation*}
  \dom(T^*)
  \dfn
  \defset{y \in Y}
  {x \mapsto \iprod[Y]{T x}{y} \text{ is continuous at every } x \in \dom(T)}
\end{equation*}
and for each $y \in \dom(T^*)$ determining $T^* y$ by the relation
\begin{equation*}
  \iprod[Y]{T x}{y} = \iprod[X]{x}{T^* y} \qfor x \in \dom(T).
\end{equation*}
We note that for any densely defined $T\: \dom(T) \subseteq X \to Y$,
the adjoint $T^*$ is closed, \cfcite[Thm.\ 13.9]{Rudin:1973},
that $T^*$ is densely defined and $T^{**} = T$
if $T$ is closed in addition, \cfcite[Thm.\ 13.12]{Rudin:1973},
and further that for $T_1\: \dom(T_1) \subseteq X \to Y$
and $T_2\: \dom(T_2) \subseteq Y \to X$,
\begin{equation*}
  \iprod[Y]{T_1 x}{y} = \iprod[X]{x}{T_2 y} \qfor x \in \dom(T_1)
\end{equation*}
implies $T_1^* \supseteq T_2$, i.e., $T_1^*$ is an extension of $T_2$
in \ifArxiv the sense \fi that $\graph(T_2) \subseteq \graph(T_1^*)$.
We refer to \cite[Ch.\ 13]{Rudin:1973} for further information
on unbounded linear operators.

\begin{lemma}
  \label{lem:closed-op}
  Let $T\: \dom(T) \subseteq X \to Y$ be a closed linear operator,
  where $X, Y$ are Hilbert spaces.
  Equip $\dom(T)$ with the graph norm of $T$,
  $\norm[T]{x}^2 \dfn \norm[X]{x}^2 + \norm[Y]{T x}^2$,
  in order to make it a Hilbert space.
  Let $S\: \dom(S) \subset X \to Y$ be a restriction of $T$.
  Then:
  \begin{enumerate}
  \item The closure of $S$ is $\cl S = T|_{\cl{\dom(S)}}$,
    where $\cl{\dom(S)}$ is the closure of $\dom(S)$
    in the graph norm of $T$.
  \item $S$ is closed if and only if
    $\dom(S)$ is a closed subspace of $\dom(T)$ in the graph norm of $T$.
  \end{enumerate}
\end{lemma}

\begin{proof}
  See \cite[Lem.~6.1]{Kurula_Zwart:2012}.
\end{proof}

\begin{lemma}
  \label{lem:inequalities}
  \LetO
  Then, the following properties hold:
  \begin{enumerate}
  \item (First Korn's inequality\ifArxiv
    \else: \fi \cite[Thm.~3.5]{Lewicka:2023})
    There exists $C > 0$ such that
    \begin{equation*}
      \norm[\LtwoM]{\grad u}^2 \le
      C \bigl( \norm[\LtwoR]{u}^2 + \norm[\LtwoS]{\Grad u}^2 \bigr)
      \qfor u \in \Hone.
    \end{equation*}
  \item It holds $\norm[\Md]{\frac12 (A + A\tp)} \le \norm[\Md]{A}$
    for every $A \in \Md$, and hence
    \begin{equation*}
      \norm[\LtwoS]{\Grad u} \le \norm[\LtwoM]{\grad u}
      \qfor u \in \Hone.
    \end{equation*}
  \end{enumerate}
\end{lemma}

\begin{remark}
  \label{rem:closed-op}
  Note that $\norm[\Hone]{}$ is the graph norm of
  $\grad\: \Hone \subset \LtwoR \to \LtwoM$
  and $\norm[\Hdiv]{}$ is the graph norm of
  $\div\: \Hdiv \subset \LtwoS \to \LtwoR$.
  Furthermore, $\norm[\Hone]{}$ is equivalent to the graph norm of
  $\Grad\: \Hone \subset \LtwoR \to \LtwoS$ by \cref{lem:inequalities}.
  Thus, \cref{lem:closed-op} (ii) provides
  that $\div, \grad$ and $\Grad$ are closed operators
  since $\Hdiv$ and $\Hone$ are complete spaces in the respective graph norms.
\end{remark}

\begin{theorem}[Duality of $\div$ and $\Grad$]
  \label{thm:div-grad-duality}
  \LetODN
  Let $G$ be a closed subspace of $\Hone$ that contains $\Honez$ and let
  \begin{equation*}
    D \dfn \defset{s \in \Hdiv}{\bprod{T_\nu s}{T_0 u} = 0 \tfor u \in G}.
  \end{equation*}
  Then, the following properties hold:
  \begin{enumerate}
  \item The set $D$ is a closed subspace of $\Hdiv$ that contains $\Hdivz$.
  \item The unbounded \ifArxiv\else linear \fi operators
    $\Grad|_G\: G \subseteq \LtwoR \to \LtwoS$ and
    $\div|_D\: D \subseteq \LtwoS \to \LtwoR$ satisfy
    \begin{equation*}
      \Grad|_G^* = -\div|_D
      \qtextq{and}
      -\div|_D^* = \Grad|_G.
    \end{equation*}
  \item $G = \HoneD \implies D = \HdivN$.
  \end{enumerate}
\end{theorem}

\begin{proof}
  We follow the proof for the scalar-vector case
  of \cite[Thm.~6.2]{Kurula_Zwart:2012},
  denoting by $\dist$ the space of distributions
  associated with $\test$ and making suitable modifications
  for the (symmetric) vector-tensor case.
  We prove (ii), then (iii), then~(i).
  \begin{enumerate}
  \item[(ii)]
    Since $\Honez \subseteq G$ and $\Honez$ is dense in $\LtwoR$,
    $G$ is dense in $\LtwoR$. Let $s$ be any element of $D$.
    From \cref{thm:div-grad-integral} and by the definition of $D$,
    we then have
    \begin{equation*}
      \iprod[\LtwoR]{-\div s}{u} = \iprod[\LtwoS]{s}{\Grad u} \qfor u \in G.
    \end{equation*}
    Thus, $-\div|_D \subseteq \Grad|_G^*$.
    Now take $s \in \dom(\Grad|_G^*) \subset \LtwoS$.
    Since $\dom(\Grad|_G) = G$ and hence
    \begin{math}
      \dom(\Grad|_G^*)
      =
      \defset{t \in \LtwoS}
      {G \owns v \mapsto \iprod[\LtwoS]{t}{\Grad v} \text{ is bounded}},
    \end{math}
    there exists $C > 0$ such that for all $v \in \test \subset G$,
    \begin{equation*}
      \abs{\iprod[\LtwoS]{s}{\Grad v}} \le C \norm[\LtwoR]{v}.
    \end{equation*}
    By density of $\test$, the bounded linear map
    $v \mapsto \iprod[\LtwoS]{s}{\Grad v}$
    can hence be extended to $\LtwoR$ and, by the Riesz--Fréchet theorem,
    there exists $u \in \LtwoR$ such that every $v \in \test$ satisfies
    \begin{equation*}
      \iprod[\LtwoR]{u}{v}
      =
      \iprod[\LtwoS]{s}{\Grad v}
      =
      \sprod[\dist, \test]{s}{\grad v}
      ,
    \end{equation*}
    where the last step uses that $s = s\tp$.
    Hence, $\div s = -u \in \LtwoR$ in the distributional sense,
    which implies that $s \in \Hdiv$.
    From \cref{thm:div-grad-integral}, we have
    \begin{equation*}
      \bprod{T_\nu s}{T_0 u} = 0,
    \end{equation*}
    which implies $s \in D$ and $\Grad|_G^* s = u = -\div s$.
    Thus, $\Grad|_G^* \subseteq -\div|_D$.
    The bidirectional inclusion proves the first equality.
    Since $\Grad|_G$ is closed by \cref{lem:closed-op} (ii)
    and Remark \labelcref{rem:closed-op},
    we \ifArxiv also \fi have $ \Grad|_G = \Grad|_G^{**} = -\div|_D^*$.
  \item[(iii)]
    Let $G = \HoneD$. By definition, $D \subseteq \HdivN$.
    Let $s \in \HdivN$ and $u \in G$.
    Then $\bprod{T_\nu s}{T_0 u} = 0$ according to
    Remark \labelcref{rem:Neumann-meaning},
    and hence also $\HdivN\subseteq D$.
  \item[(i)]
    Clearly, $\Hdivz = \ker T_s \subseteq D \subseteq \Hdiv \subset \LtwoS$.
    \ifArxiv Now, the adjoint \else The adjoint \fi operator
    $\Grad|_G^* = -\div|_D$ established in (ii) is necessarily closed.
    Therefore, assertion (ii) of \cref{lem:closed-op}
    applied to $S = -\div|_D$ and $T = -\div$
    implies that $D = \dom(S)$ is a closed subspace of $\Hdiv = \dom(T)$.
    \qedhere
  \end{enumerate}
\end{proof}

\begin{remark}
  Our claim (iii) in \cref{thm:div-grad-duality}
  is more general than in \Cite{Kurula_Zwart:2012} where only
  \ifArxiv the cases \fi
  $G = H_0^1(\Omega)$ and $G = H^1(\Omega)$ are treated.
  Moreover, our proof is simpler
  because \cite[Thm.\ 6.2.3]{Kurula_Zwart:2012}
  also shows that $G$ is uniquely determined by $D$.
  This is irrelevant, however: it suffices that
  $\bdD$ and $\bdN$ determine $G$ and (in turn) $D$ as given.
\end{remark}

Throughout the remainder of this \ifArxiv work\else contribution\fi,
we consider $G \dfn \HoneD$ and $D \dfn \HdivN$ with
\begin{equation*}
  B \dfn \Grad|_G \qtextq{and} B^* = -\div|_D,
\end{equation*}
including the extreme cases
of pure Dirichlet boundary data
and pure Neumann boundary data,
\begin{gather*}
  \Gamma = \bdD \implies G = \Honez \implies D = \Hdiv, \\
  \Gamma = \bdN \implies G = \Hone \implies D = \Hdivz.
\end{gather*}
We write $G$ and $D$ when regarding \ifArxiv them \else these spaces \fi
as sets or as dense subspaces of $\LtwoR$ or $\LtwoS$, respectively,
and we write $\Honez, \HoneD, \Hone$ and $\Hdivz, \HdivN, \Hdiv$
when regarding them as Hilbert spaces with the respective norms
$\norm[\Hone]{}$ and $\norm[\Hdiv]{}$.
(As an exception, the \ifArxiv orthogonal \fi decompositions
of \cref{lem:ext-D,lem:ext-N} will be denoted as
$\Hone = G \oplus G^\perp$ and $\Hdiv = D \oplus D^\perp$.)
In all cases we have an orthogonal decomposition of $\LtwoS$,
for which we need the following definition and lemma;
see also \cite[Lem.~3.2]{Lewicka:2023}.

\begin{definition}
  The space $R$ of infinitesimal rigid body motions in $\Rd$ is
  \begin{equation*}
    R \dfn \defset{u \in \LtwoR}{\Grad u = 0}
    =
    \defset{(x \mapsto A x + b)}{(A, b) \in \Ad \times \Rd}.
  \end{equation*}
  Its orthogonal complement in $\LtwoR$ will be denoted as $Q$.
\end{definition}

\begin{remark}
  Clearly, $R$ is isomorphic to the tangent space $T(SO(d) \times \Rd)$
  of dimension $\frac12 d (d + 1)$,
  we have $R \subset C^\oo(\Omega, \Rd)$,
  and both $R$ and $Q$ are closed.
  The tangent property of $R$ explains why the symmetric gradient
  generates stresses that are independent of rigid rotations and translations.
  For Korn's inequality (next lemma) and
  later on for pure Neumann boundary data,
  restricting displacements to $Q$ is essential.
\end{remark}

\begin{lemma}
  \label{lem:P+K}
  \LetO
  Then, the inequalities of Poincaré (with $C_1 > 0$) and Korn
  (with $C_2 > 0$) hold for every $u \in \Hone \cap Q\:$
  \begin{equation*}
    \norm[\LtwoR]{u}^2
    \le C_1 \norm[\LtwoM]{\grad u}^2
    \le C_1 C_2 \norm[\LtwoS]{\Grad u}^2.
  \end{equation*}
\end{lemma}

\begin{proof}
  See steps 1.\ and 2.\ in the proof of \cite[Thm.~3.8]{Lewicka:2023}.
\end{proof}

\begin{lemma}[Helmholtz decomposition]
  \label{lem:M+N=L2}
  \LetODN
  Let
  \begin{equation*}
    M \dfn \defset{\Grad u}{u \in G} = \im B
  \end{equation*}
  and
  \begin{equation*}
    N \dfn \defset{s \in D}{\div s = 0} = \ker B^*.
  \end{equation*}
  Then, $M$ and $N$ are closed in $\LtwoS$, $M \oplus N = \LtwoS$,
  and the decomposition is orthogonal.
\end{lemma}

\begin{proof}
  This follows from $(\im B)^\perp = \ker B^*$ since $M = \im B$ is closed,
  as we now prove.
  Let $(s_n)_n \subset M$ converge to $s$ in $\LtwoS$.
  For each $n$, let $s_n = \Grad u_n$ for some $u_n \in G \cap Q$.
  Then, $(u_n)_n$ is bounded in $G = \HoneD$ by \cref{lem:P+K}
  because $\Grad u_n = s_n$ is bounded in $\LtwoS$.
  Hence, there exists $u \in G$ such that $u_{n_k} \rightharpoonup u$
  weakly in $\HoneD$ and consequently in $\LtwoR$ for a subsequence.
  Thus, every test function $t \in \test[\Sd]$ satisfies
  \begin{align*}
    \iprod[\LtwoS]{s}{t}
    &= \lim\iprod[\LtwoS]{s_{n_k}}{t} = \lim\iprod[\LtwoS]{\Grad u_{n_k}}{t} \\
    &= \lim\iprod[\LtwoR]{u_{n_k}}{-\div t}
      = \iprod[\LtwoR]{u}{-\div t} = \iprod[\LtwoS]{\Grad u}{t},
  \end{align*}
  which gives $s = \Grad u \in M$.
\end{proof}

\begin{remark}
  A direct proof (without using adjoint\ifArxiv\ operator\fi s)
  is sketched in \cite[Lem.~2.1]{Conti2018} for the case $\bdD \ne \0$.
  Decompositions as in \cref{lem:M+N=L2} can also exist
  for unbounded $\Omega$ and in $L^r$ with $1 < r < \oo$;
  see \Cite{Hieber_et_al:2021} and references therein.
\end{remark}

\section{Data-Driven Elasticity Problems}
\label{sec:dd-problems}

Let $\data$ be a given data set of strain-stress pairs
that is \emph{closed} in $\Sd \times \Sd$.
To formulate the basic data-driven problems,
let $Z \dfn \LtwoS \times \LtwoS$ and
construct auxiliary fields $\ed, \sd$ from $\data$ as
\begin{equation*}
  \setD_d \dfn \defset{(\ed, \sd) \in Z}
  {(\ed(x), \sd(x)) \in \data \tfor x \in \Omega}.
\end{equation*}
Consider also the special case where $\data$ is defined
by some constitutive relationship $\cg\: \Sd \times \Sd \to \Sd$,
\begin{equation*}
  \setD_c \dfn \defset{(\ed, \sd) \in Z}
  {\cg(\ed(x), \sd(x)) = 0 \tfor x \in \Omega}.
\end{equation*}
For small deformations,
the standard elasticity problem in Hilbert space
is a PDE boundary value problem
wherein the strain $e$ and the stress $s$ are coupled
by such a material law $\cg(e, s) = 0$,
leading to a (generally nonlinear) elliptic problem.

\begin{problem}[\BPM0]
  Let $\XM0 \dfn \Hone \times \LtwoS \times \Hdiv$.
  Given a force field $f \in \LtwoR$ and
  boundary data $g \in \Hphalf$, $h \in \Hnhalf$,
  consider for $x = (u, e, s)$ in $\XM0$ the problem
  \begin{align*}
    e - \Grad u &= 0, \quad T_0 u = g \on \bdD, \\
    \div s + f &= 0, \quad T_\nu s = h \on \bdN, \\
    \cg(e, s) &= 0.
  \end{align*}
\end{problem}

\begin{remark}
  If there exists a hemicontinuous map $\cG\: \LtwoS \to \LtwoS$
  that satisfies $\cg(e, \cG(e)) = 0$
  \ifArxiv
  and that is monotone and coercive in the sense that
  $\iprod[\LtwoS]{\cG(e_2) - \cG(e_1)}{e_2 - e_1} \ge 0$ respectively
  $\iprod[\LtwoS]{\cG(e)}{e} \ge C \norm[\LtwoS]{e}^2$ with $C > 0$,
  \else
  and that is monotone,
  $\iprod[\LtwoS]{\cG(e_2) - \cG(e_1)}{e_2 - e_1} \ge 0$,
  and coercive,
  $\iprod[\LtwoS]{\cG(e)}{e} \ge C \norm[\LtwoS]{e}^2$ with $C > 0$,
  \fi
  then the Minty--Browder theorem \cite{Browder:1967}
  guarantees that \BPM0 admits a solution.
  That solution is unique if $\cG$ is strictly monotone.
\end{remark}

\begin{remark}
  Note that we require $g$ and $h$
  to be defined on the entire boundary
  to avoid technical issues with the regularity of
  $\bd\bdD \cap \bd\bdN \subseteq \Gamma \setminus (\bdD \cup \bdN)$.
  While locally defined boundary data $h_N \in \Hnhalf[N]$
  can simply be extended by zero to obtain $h \in \Hnhalf$,
  $g_D \in \Hphalf[D]$ may not admit an extension to $\Hphalf$
  for certain exotic geometries;
  see also the discussion in \cite[Rem.\ 3.1]{Gudoshnikow_Krizk:2025}.
\end{remark}

Data-driven problems, cf.\
\Cites{Kirchdoerfer2016,Kirchdoerfer2018,Conti2018,Conti2020,%
  Gebhardt_Lange_Steinbach:2025,Gebhardt_Steinbach:2025},
relax the hard coupling $\cg(e, s) = 0$ by minimizing a suitable distance
of $(e, s)$ to $\setD_d$ (or $\setD_c$), here the $L^2$ distance
\begin{equation*}
  \obj(e - \ed, s - \sd)
  \dfn
  \frac{c}{2} \norm[\LtwoS]{e - \ed}^2 +
  \frac{1}{2 c} \norm[\LtwoS]{s - \sd}^2,
\end{equation*}
where $c$ is a positive constant that provides unit consistency.
Note that the formally more general distance
$\norm[\C]{e - \ed}^2 + \norm[\Inv\C]{s - \sd}^2$
with a rank-4 tensor $\C$ that maps $\LtwoS$ to itself
would not introduce additional complexity.
Note further that physical units of $u$ and $s$
in $\norm[\Hone]{u}$ and $\norm[\Hdiv]{s}$
also necessitate a suitable scaling to ensure unit consistency.
This has no impact on our analysis, however,
as it produces equivalent norms.
We begin the analysis with the data-driven counterpart of \BPM0.

\begin{problem}[\BPM1]
  Let $\XM1 \dfn \XM0$.
  Given $f \in \LtwoR$, $g \in \Hphalf$ and $h \in \Hnhalf$,
  consider for $x = (u, e, s)$ in $\XM1$ and $z = (\ed, \sd)$ in $Z$
  the optimization problem
  \begin{align*}
    \smash[b]{\inf_{(x, z) \in \XM1 \times Z} \ \obj(e - \ed, s - \sd)}
    \qstq
    e - \Grad u &= 0, \quad T_0 u = g \on \bdD, \\
    \div s + f &= 0, \quad T_\nu s = h \on \bdN, \\
    (\ed, \sd) &\in \setD,
  \end{align*}
  where $\setD$ is the set of auxiliary fields,
  either $\setD_c$ or $\setD_d$.
  (Do not confuse $\setD$ with \ifArxiv the test functions $\mathscr D$\else
  $\mathscr D$, the space of test functions\fi).
\end{problem}

For a more concise formulation,
we now express the solution fields as orthogonal sums
$u = u_0 + \extD g$ and $s = s_0 + \extN h$ with unknowns
$u_0 \in G$ and $s_0 \in D$, respectively,
while the orthogonal components
$\extD g \in G^\perp$ and $\extN h \in D^\perp$
are uniquely given by the boundary data
due to \cref{lem:ext-D,lem:ext-N}.
The associated set of shifted auxiliary fields is
$\setD_0 \dfn \defset{(\ed, \sd_0)}
{(\ed, \sd_0 + \extN h) \in \setD}$
with $\setD = \setD_c$ or $\setD = \setD_d$.
Recall that we have $B = \Grad|_G$ and $B^* = -\div|_D$
with $G = \HoneD$ and $D = \HdivN$.

\begin{problem}[\BPM2] \label{labelBPM2}
  Let $\XM2 \dfn G \times \LtwoS \times D$.
  Given fields $f \in \LtwoR$, $g \in \Hphalf$ and $h \in \Hnhalf$,
  let $p \dfn \Grad \extD g$ and $q \dfn f + \div \extN h$,
  and consider for
  $x = (u_0, e, s_0)$ in $\XM2$ and $z = (\ed, \sd_0)$ in $Z$
  the optimization problem
  \begin{align*}
    \smash[b]{\inf_{(x, z) \in \XM2 \times Z} \ \obj(e - \ed, s_0 - \sd_0)}
    \qstq
    e - B u_0 - p &= 0 \in \LtwoS, \\
    B^* s_0 - q &= 0 \in \LtwoR, \\
    (\ed, \sd_0) &\in \setD_0.
  \end{align*}
\end{problem}

\begin{proposition}
  Problems \BPM1 and \BPM2 are equivalent.
\end{proposition}

\begin{proof}
  Given $(u_0, e, s_0, \ed, \sd_0)$ in $\XM2 \times Z$,
  let $u \dfn u_0 + \extD g \ifArxiv$ in $\else\in\fi G \oplus G^\perp = \Hone$
  and $s \dfn s_0 + \extN h \ifArxiv$ in $\else\in\fi D \oplus D^\perp = \Hdiv$,
  similarly $\sd \dfn \sd_0 + \extN h \in \LtwoS$.
  Then, $(u, e, s, \ed, \sd)$ in $\XM1 \times Z$
  is feasible for \BPM1 if and only if
  $(u_0, e, s_0, \ed, \sd_0)$ is feasible for \BPM2.
  Indeed, the definition of $\setD_0$ gives
  $(\ed, \sd) \in \setD$ iff
  $(\ed, \sd_0) \in \setD_0$,
  and the definitions of $p, q, u, s$
  combined with \cref{lem:ext-D,lem:ext-N} give
  \ifArxiv
  \begin{align*}
    e - \Grad u &= e - \Grad (u_0 + \extD g) = e - B u_0 - p,
    &T_0 u|_{\bdD} = T_0 u_0|_{\bdD} + T_0 \extD g|_{\bdD} &= g|_{\bdD}, \\
    \div s + f &= \div (s_0 + \extN h) + f = -B^* s_0 + q,
    &T_\nu s|_{\bdN} = T_\nu s_0|_{\bdN} + T_\nu \extN h|_{\bdN} &= h|_{\bdN}.
  \end{align*}
  \else
  \begin{align*}
    e - \Grad u &= e - \Grad (u_0 + \extD g) = e - B u_0 - p, \\
    \div s + f &= \div (s_0 + \extN h) + f = -B^* s_0 + q, \\
    T_0 u|_{\bdD} &= T_0 u_0|_{\bdD} + T_0 \extD g|_{\bdD} = g|_{\bdD}, \\
    T_\nu s|_{\bdN} &= T_\nu s_0|_{\bdN} + T_\nu \extN h|_{\bdN} = h|_{\bdN}.
  \end{align*}
  \fi
  Conversely, given $(u, e, s, \ed, \sd)$ in $\XM1 \times Z$ solving \BPM1,
  let $u_0 \dfn u - \extD g$, $s_0 \dfn s - \extN h$
  and $\sd_0 \dfn \sd - \extN h$ to recover
  $(u_0, e, s_0, \ed, \sd_0)$ in $\XM2 \times Z$.
  Finally, the objective values coincide since $s - \sd = s_0 - \sd_0$.
\end{proof}

Without loss of generality, we set $h \dfn 0$
for pure Dirichlet boundary data from now on:
$D^\perp = \set{0}$ implies $\extN h = 0$ and $q = f$ for any $h$.
Similarly, we set $g \dfn 0$ for pure Neumann boundary data:
$G^\perp = \set{0}$ implies $\extD g = 0$ and $p = 0$ for any $g$.

\begin{proposition}
  \label{prop:inj-surj}
  In problem \BPM2 with $\bdD \ne \0$,
  $B$ is injective and $B^*$ is surjective onto $\LtwoR$.
  For $\bdD = \0$ (pure Neumann boundary data), the operator
  $B_N \dfn B|_{G \cap Q}\: G \cap Q \subset Q \to \LtwoS$
  is injective with $\im B_N = \im B = M$,
  and $B_N^*$ is surjective onto $Q$ with $\ker B_N^* = \ker B^* = N$.
\end{proposition}

\begin{proof}
  By the definitions of $B$ and $R$, we have $\ker B = G \cap R$.
  Then, firstly, let $\bdD \ne \0$.
  In this case, $G \cap R = \set{0}$, \cite[Cor.\ 3.10]{Lewicka:2023},
  hence $B$ is injective.
  By \cref{lem:M+N=L2}, $\im B$ is closed;
  according to the Closed Range Theorem, so is $\im B^*$.
  Surjectivity of $B$ hence follows from
  $(\im B^*)^\perp = \ker B^{**} = \ker B = \set{0}$.
  Also from \cref{lem:M+N=L2}, we have $\im B = M$, $\ker B^* = N$.
  If $\bdD = \0$, $\ker B = G \cap R = R$
  and hence $\ker B_N = \set{0}$.
  Now $B_N$---being a restriction of a closed operator to a closed subspace---%
  is again closed, and since $B = 0$ on $G \cap Q^\perp = G \cap R$,
  $\im B_N = \im B$.
  Hence, again, $\im B_N$ and $\im B_N^*$ are closed and
  $\im B_N^* = \cl{\im B_N^*} = (\ker B_N)^\perp = \set{0}^\perp = Q$.
\end{proof}

\begin{remark}
  Note that in \cref{prop:inj-surj}, we need to regard $B_N$ as operator
  $B_N\: G \cap Q \subset Q \to \LtwoS$ and not
  $B_N\: G \cap Q \subset \LtwoR \to \LtwoS$.
  This is necessary for $B_N$ to still be densely defined
  and hence for its adjoint to be well-defined.
\end{remark}

Because of \cref{prop:inj-surj},
\BPM2 with pure Neumann boundary data has inconsistent constraints unless
$f$ and $h$ obey a compatibility condition
related to the space $R^\perp = Q$.
In specific, with $G_Q \dfn G \cap Q$, \BPM2 specializes as follows,
where $G = \Hone$ and $D = \Hdivz$ with $g = 0$ and $p = 0$.

\begin{problem}[\BPN2]
  Let $\XN2 \dfn G_Q \times \LtwoS \times D$.
  Given $f \in \LtwoR$ and $h \in \Hnhalf$,
  let $q \dfn f + \div \extN h$ and assume
  that the compatibility condition $q \in  Q$ holds.
  Consider for $x = (u, e, s_0)$ in $\XN2$ and $z = (\ed, \sd_0)$ in $Z$
  the optimization problem
  \begin{align*}
    \smash[b]{\inf_{(x, z) \in \XN2 \times Z} \ \obj(e - \ed, s_0 - \sd_0)}
    \qstq
    e - B_N u &= 0 \in \LtwoS, \\
    B_N^* s_0 - q &= 0 \in Q, \\
    (\ed, \sd_0) &\in \setD_0.
  \end{align*}
\end{problem}

\begin{proposition}
  \label{prop:low-reg-B-B*}
  For $\bdN = \0$ (pure Dirichlet boundary data),
  we have bounded linear operators
  \begin{align*}
    B_D \dfn \Grad\: \Honez &\to \LtwoS,
    &B_D^* = -\div\: \LtwoS & \to \Honez^*,
  \end{align*}
  where \ifArxiv only \fi $\LtwoS$ is identified
  with its dual \ifArxiv space \fi but not $\Honez$,
  and $-B_D^*$ is the distributional divergence, defined via
  $\sprod[\Honez^*,\Honez]{B_D^* s}{u} \dfn \iprod[\LtwoS]{s}{B_D u}$.
  Again, $B_D$ is injective with $\im B_D = \im B = M$ and the
  adjoint $B_D^*$ is surjective with $\ker B_D^* = \ker B^* = N$.
\end{proposition}

\begin{proof}
  Clear from the definitions.
\end{proof}

Due to \cref{prop:low-reg-B-B*}, we have an additional low-regularity
version of \BPM2 for pure Dirichlet boundary data,
with \ifArxiv body \fi force $f \in \Honez^*$ and stress $s \in \LtwoS$.
This \ifArxiv is the setting that \else setting \fi
was studied in \Cite{Gebhardt_Steinbach:2025}
for the case $d = 1$ using the Gelfand triple
$H_0^1(\Omega) \subset \Ltwo \subset H_0^1(\Omega)^*$.
Independent of the regularity, this problem has Neumann boundary data $h = 0$
and lacks the shifts of $s, \sd, \setD$ since $q = f + \div \extN h = f$.

\begin{problem}[\BPD2]
  Let $\XD2 \dfn \Honez \times \LtwoS \times \LtwoS$.
  Given fields $f \in \Honez^*$ and $g \in \Hphalf$,
  let $p \dfn \Grad \extD g$ and consider for
  $x = (u_0, e, s)$ in $\XD2$ and $z = (\ed, \sd)$ in $Z$
  the optimization problem
  \begin{align*}
    \smash[b]{\inf_{(x, z) \in \XD2 \times Z} \ \obj(e - \ed, s - \sd)}
    \qstq
    e - B_D u_0 - p &= 0 \in \LtwoS, \\
    B_D^* s - f &= 0 \in \Honez^*, \\
    (\ed, \sd) &\in \setD.
  \end{align*}
\end{problem}

\begin{remark}
  The reformulation \BPM2 of \BPM1 along with its special variants
  \BPN2 and \BPD2 can similarly be applied
  to the standard elasticity problem \BPM0.
  In contrast, the following material is generally specific
  to the data-driven problems
  and applies to \BPM0 only if
  \BPM1 admits a \emph{strong solution},
  i.e., $(e, s) = (\ed, \sd) \in \setD$ with $\inf\obj = \min\obj = 0$,
  so that \BPM0 and \BPM1 are equivalent.
\end{remark}
In the further analysis, we fix the auxiliary fields $\ed, \sd_0$
and solve the resulting continuous subproblems.
To this end, consider the Lagrangian of \BPM2 (with $\ed, \sd_0$ fixed),
\ifArxiv
\begin{equation*}
  L(u_0, e, s_0, \lambda, \mu) =
  \obj(e - \ed, s_0 - \sd_0)
    + \frac1c \iprod[\LtwoR]{\lambda}{B^* s_0 - q}
    + c \iprod[\LtwoS]{\mu}{e - B u_0 - p}
    ,
\end{equation*}
\else
\begin{multline*}
  L(u_0, e, s_0, \lambda, \mu) = {} \\
  \obj(e - \ed, s_0 - \sd_0)
    + \smash[t]{\frac1c} \iprod[\LtwoR]{\lambda}{B^* s_0 - q}
    + c \iprod[\LtwoS]{\mu}{e - B u_0 - p}
    ,
\end{multline*}
\fi
wherein the dual fields (Lagrange multipliers) $\lambda$ and $\mu$
are scaled to match corresponding objective terms.
We begin with Dirichlet or mixed boundary data.

\begin{theorem}
  \label{prop:fixed-z}
  Let $\bdD \ne \0$.
  Then, fixing any pair $z = (\ed, \sd_0) \in \setD_0$ in \BPM2
  yields a unique minimizer $x^* = (u_0^*, e^*, s_0^*)$ in $\XM2$
  and unique dual fields $\lambda^*$ in $G$
  and $\mu^*$ in $N \subset D$,
  all of which depend linearly on $z$ and $f, g, h$.
  Explicit expressions are given in \cref{eq:BPM2-x,eq:BPM2-y} in the proof.
\end{theorem}

\begin{proof}
  With $z = (\ed, \sd_0)$ in $\setD_0$ fixed, \BPM2 reduces to
  \begin{align*}
    \smash[b]{\inf_{x \in \XM2} \ \obj(e - \ed, s_0 - \sd_0)}
    \qstq
    e - B u_0 - p &= 0 \in \LtwoS, \\
    B^* s_0 - q &= 0 \in \LtwoR.
  \end{align*}
  This equality-constrained quadratic optimization problem
  \ifArxiv
  admits a unique global minimizer since
  the linear constraint map $x \mapsto (e - B u_0, B^* s_0)$ is surjective
  with kernel $K = \defset{(u_0, B u_0, s_0)}{(u_0, s_0) \in G \times N}$
  and $\obj(e, s_0)$ is coercive on $K$.
  Indeed, since $G \subset Q$ by \cite[Cor. 3.10]{Lewicka:2023},
  this follows from \cref{lem:P+K}: letting
  \else
  admits a unique global minimizer. Indeed,
  the linear constraint map $x \mapsto (e - B u_0, B^* s_0)$ is surjective
  with kernel $K = \defset{(u_0, B u_0, s_0)}{(u_0, s_0) \in G \times N}$,
  and $\obj(e, s_0)$ is coercive on $K$ by \cref{lem:P+K} since
  $G \subset Q$ by \cite[Cor. 3.10]{Lewicka:2023}: letting
  \fi
  $C \dfn \min\bigl(\frac{c}{4 C_1 C_2}, \frac{c}{4}, \frac{1}{2 c} \bigr)$
  gives
  \begin{align*}
    \frac{c}{2} \norm[\LtwoS]{B u_0}^2 +
    \frac{1}{2 c} \norm[\LtwoS]{s_0}^2
    &\ge C \bigl(\norm[\LtwoS]{u_0}^2
      + \norm[\LtwoS]{B u_0}^2
      + \norm[\LtwoS]{s_0}^2 \bigr).
  \end{align*}
  Moreover, the KKT conditions obtained by setting
  the partial Frechét derivatives of $L$ to zero
  (see, e.g., \cite[Thm.~3.2]{Maurer_Zowe:1979})
  uniquely determine the minimizer $(u_0^*, e^*, s_0^*)$
  along with multipliers $\lambda^*$ in $\LtwoR$ and $\mu^*$ in $\LtwoS$:
  \begin{align}
    \label{eq:kkt1}
    \iprod[\LtwoS]{\mu}{B \delta u} &= 0 \qfor \delta u \in G, \\
    \label{eq:kkt2}
    \iprod[\LtwoS]{e - \ed + \mu}{\delta e} &= 0 \qfor \delta e \in \LtwoS, \\
    \label{eq:kkt3}
    \iprod[\LtwoS]{s_0 - \sd_0}{\delta s} +
    \iprod[\LtwoR]{\lambda}{B^* \delta s} &= 0 \qfor \delta s \in D, \\
    B^* s_0 - q &= 0, \notag \\
    e - B u_0 - p &= 0. \notag
  \end{align}
  Condition \eqref{eq:kkt2} is \ifArxiv obviously \else clearly \fi
  equivalent to $e - \ed + \mu = 0$.
  By \cref{lem:M+N=L2},
  \eqref{eq:kkt1} \ifArxiv implies \else gives \fi $\mu \in M^\perp = N$,
  and hence by \cref{prop:inj-surj},
  \eqref{eq:kkt1} is equivalent to $B^* \mu = 0$.
  Similarly, taking $\delta s$ in $\ker B^* = N \subset D$
  gives $s_0 - \sd_0 \in M = \im B$ by \eqref{eq:kkt3}, i.e.,
  $s_0 - \sd_0 = B \lambda_0$ for some $\lambda_0 \in G$,
  making \eqref{eq:kkt3} equivalent to
  $\iprod[\LtwoR]{\lambda_0 + \lambda}{B^* \delta s} = 0$
  for all $\delta s \in D$.
  Surjectivity of $B^*$ now gives $\lambda = - \lambda_0 \in G$,
  and \eqref{eq:kkt3} is finally equivalent to
  $s_0 - \sd_0 + B \lambda = 0$.
  Thus, fixing $(\ed, \sd_0)$ splits \BPM2 into two problems,
  one for $u_0$ and $e$ with the multiplier $\mu$ in $N \subset D$,
  \begin{align*}
    B^* \mu &= 0 \in \LtwoR,
    & e + \mu &= \ed \in \LtwoS,
    & e - B u_0 &= p \in \LtwoS,
  \end{align*}
  and another one for $s_0$ with the multiplier $\lambda$ in $G$,
  \begin{align*}
    s_0 + B \lambda &= \sd_0 \in \LtwoS,
    & B^* s_0 &= q \in \LtwoR.
  \end{align*}
  To conclude the \ifArxiv construction\else proof\fi,
  we regard $B$ and $B^*$ as bounded operators
  on $\HoneD$ and $\HdivN$, respectively,
  so that \cref{lem:psi} \ifArxiv provides \else gives \fi pseudoinverses
  $B^+\: \LtwoS \to \HoneD$ (surjective) and
  $(B^*)^+\: \LtwoR \to \HdivN$ (injective with range $M \cap \HdivN$).
  In terms of these pseudoinverses and the orthogonal projections $P_M, P_N$,
  the unique solutions are readily obtained as
  \ifArxiv
  \begin{equation}
    \label{eq:BPM2-x}
    e^* = P_M \ed + P_N p, \qquad
    \mu^* = \ed - e^* = P_N (\ed - p), \qquad
    u_0^* = B^+ (e^* - p) = B^+ (\ed - p),
  \end{equation}
  respectively
  \begin{equation}
    \label{eq:BPM2-y}
    s_0^* = P_N \sd_0 + (B^*)^+ q, \qquad\qquad
    \lambda^* = B^+ (\sd_0 - s_0^*) = B^+ \sd_0 - B^+ (B^*)^+ q.
    \qedhere
  \end{equation}
  \else
  \begin{align}
    \label{eq:BPM2-x}
    e^* &= P_M \ed + P_N p,
    &\mu^* &= \ed - e^* = P_N (\ed - p),
    &u_0^* = B^+ (e^* - p) = B^+ (\ed - p),
  \end{align}
  and
  \begin{align}
    \label{eq:BPM2-y}
    s_0^* &= P_N \sd_0 + (B^*)^+ q,
    &\lambda^* &= B^+ (\sd_0 - s_0^*) = B^+ \sd_0 - B^+ (B^*)^+ q,
  \end{align}
  respectively.
  \fi
\end{proof}

\begin{remark}
  The solution expressions are identical to
  \cite[Thm.~1]{Gebhardt_Steinbach:2025},
  except that we have $B^+ (B^*)^+$ rather than $\Inv{(B^* B)}$ in $\lambda^*$.
  This is because $B$ and $B^*$ are adjoints
  only in the sense of unbounded operators,
  and the two last expressions in \cref{lem:psi}
  ($A$ injective or surjective) do not apply to $B$ and $B^*$.
  In fact, $\dom(B^* B)$ is a proper subspace of $G$
  and $B^* B$ is not even surjective onto $\LtwoR$.
\end{remark}

\begin{theorem}[Pure Neumann boundary data]
  \label{prop:fixed-zN}
  Let $\bdD = \0$.
  Then, fixing any pair $z = (\ed, \sd_0) \in \setD_0$ in \BPN2
  yields a unique minimizer $x^* = (u^*, e^*, s_0^*)$ in $\XN2$
  and unique dual fields $\lambda^*$ in $G_Q$
  and $\mu^*$ in $N \subset D$,
  all of which depend linearly on $z, f, h$.
  Explicit expressions are given in \cref{eq:BPN2-x,eq:BPN2-y} in the proof.
\end{theorem}

\begin{proof}
  \ifcase0
  The proof proceeds exactly as in \cref{prop:fixed-z},
  with certain quantities replaced according to \cref{tab:dd-problems}.
  This gives the following solution components,
  where $u^*, e^*, \mu^*$ simplify because $p = 0$:
  \or
  With $(\ed, \sd_0)$ in $\setD_0$ fixed, \BPN2 reduces to
  \begin{align*}
    \smash[b]{\inf_{(u_0, e, s_0) \in \XN2} \ \obj(e - \ed, s_0 - \sd_0)}
    \qstq
    e - B_N u &= 0 \in \LtwoS, \\
    B_N^* s_0 - q &= 0 \in Q.
  \end{align*}
  Again, the problem admits a unique global minimizer since
  the constraint map $x \mapsto (B_N^* s_0, e - B_N u)$ is surjective
  and $\obj(e, s_0)$ is coercive on its kernel,
  $\defset{(u, B_N u, s_0)}{(u, s_0) \in G_Q \times N}$.
  The Lagrangian now reads
  \begin{align*}
    L(u, e, s_0, \lambda, \mu)
    &= \obj(e - \ed, s_0 - \sd_0)
      + \frac1c \iprod[Q]{\lambda}{B_N^* s_0 - q}
      + c \iprod[\LtwoS]{\mu}{e - B_N u}
      ,
  \end{align*}
  and the associated KKT conditions become
  \begin{align}
    \label{eq:kktN1}
    \iprod[\LtwoS]{\mu}{B_N \delta u} &= 0 \qfor \delta u \in G_Q, \\
    \label{eq:kktN2}
    \iprod[\LtwoS]{e - \ed + \mu}{\delta e} &= 0 \qfor \delta e \in \LtwoS, \\
    \label{eq:kktN3}
    \iprod[\LtwoS]{s_0 - \sd_0}{\delta s} +
    \iprod[Q]{\lambda}{B_N^* \delta s} &= 0 \qfor \delta s \in D, \\
    B_N^* s_0 - q &= 0, \notag \\
    e - B_N u &= 0. \notag
  \end{align}
  Condition \eqref{eq:kktN2} is equivalent to $e - \ed + \mu = 0$.
  By \cref{lem:M+N=L2,prop:inj-surj},
  \eqref{eq:kktN1} implies $\mu \in M^\perp = N$,
  and hence \eqref{eq:kktN1} is equivalent to $B_N^* \mu = 0$.
  Taking $\delta s$ in $\ker B_N^* = N \subset D$
  gives $s_0 - \sd_0 \in M$ by \eqref{eq:kktN3}, i.e.,
  $s_0 - \sd_0 = B_N \lambda_0$ with $\lambda_0 \in G_Q$,
  which makes \eqref{eq:kktN3} equivalent to
  $\iprod[Q]{\lambda_0 + \lambda}{B_N^* \delta s} = 0$
  for $\delta s \in D$.
  Surjectivity of $B_N^*$ now gives $\lambda = - \lambda_0 \in G_Q$,
  and \eqref{eq:kktN3} is finally equivalent to
  $s_0 - \sd_0 + B_N \lambda = 0$.
  Thus, fixing $(\ed, \sd)$ splits \BPN2 into two problems,
  one for $u$ and $e$ with the multiplier $\mu$ in $N \subset D$,
  \begin{align*}
    B_N^* \mu &= 0 \in Q,
    & e + \mu &= \ed \in \LtwoS,
    & e - B_N u &= 0 \in \LtwoS,
  \end{align*}
  and another one for $s_0$ with the multiplier $\lambda$ in $G_Q$,
  \begin{align*}
    s_0 + B_N \lambda &= \sd_0 \in \LtwoS,
    & B_N^* s_0 &= q \in Q.
  \end{align*}
  Using orthogonal projections $P_M, P_N$
  and pseudoinverses $B_N^+\: \LtwoS \to \Hone \cap Q$ (surjective)
  and $(B_N^*)^+\: Q \to \Hdivz$ (injective with range $M \cap \Hdivz$),
  the unique solutions are readily obtained as
  \fi
  \ifArxiv
  \begin{equation}
    \label{eq:BPN2-x}
    e^* = P_M \ed, \qquad
    \mu^* = \ed - e^* = P_N \ed, \qquad
    u^* = B_N^+ e^* = B_N^+ \ed,
  \end{equation}
  respectively
  \begin{equation}
    \label{eq:BPN2-y}
    s_0^* = P_N \sd_0 + (B_N^*)^+ q, \qquad\qquad
    \lambda^* = B_N^+ (\sd_0 - s_0^*) = B_N^+ \sd - B_N^+ (B_N^*)^+ q.
    \qedhere
  \end{equation}
  \else
  \begin{align}
    \label{eq:BPN2-x}
    e^* &= P_M \ed,
    &\mu^* &= \ed - e^* = P_N \ed,
    &u^* = B_N^+ e^* = B_N^+ \ed,
  \end{align}
  and
  \begin{align}
    \label{eq:BPN2-y}
    s_0^* &= P_N \sd_0 + (B_N^*)^+ q,
    &\lambda^* &= B_N^+ (\sd_0 - s_0^*) = B_N^+ \sd - B_N^+ (B_N^*)^+ q,
  \end{align}
  respectively.
  \fi
\end{proof}

\begin{table}
  \caption{Corresponding quantities in problems \BPM2, \BPN2, \BPD2.}
  \def\PR#1{\text{BP#1}_2\:\quad}
  \ifArxiv\vspace*{-1.6ex}\fi
  \centering
  \begin{tabular}{c}
    \toprule
    $\displaystyle
    \begin{aligned}
      \PR M& G && D && \LtwoR && B && B^* &&u_0&& s_0&&\sd_0&&p&&q\\[-1pt]
      \PR N&G_Q&& D && Q  &&B_N&&B_N^*&& u && s_0&&\sd_0&&0&&q\\[-1pt]
      \PR D&\Honez&&\LtwoS&&\Honez^*&&B_D&&B_D^*&&u_0&&  s  &&\sd &&p&&f
    \end{aligned}$ \\[-1pt]
    \bottomrule
  \end{tabular}
  \label{tab:dd-problems}
\end{table}

\begin{theorem}[Pure  Dirichlet boundary data]
  \label{prop:fixed-zD}
  Let $\bdN = \0$.
  Then, fixing any pair $z = (\ed, \sd) \in \setD$
  in the low-regularity problem \BPD2
  yields a unique minimizer $x^* = (u_0^*, e^*, s^*)$ in $\XD2$
  and unique dual fields $\lambda^*$ in $\Honez^{**}\! \isom \Honez$
  and $\mu^*$ in $N \subset \Hdiv$,
  all of which depend linearly on $z, f, g$.
  Explicit expressions are given in \cref{eq:BPD2-x,eq:BPD2-y} in the proof.
\end{theorem}

\begin{proof}
  \ifcase0
  The proof proceeds again as in \cref{prop:fixed-z},
  with certain quantities replaced according to \cref{tab:dd-problems}.
  In addition, $\iprod[\LtwoR]{\fcdot}{\fcdot}$
  is replaced with the duality pairing
  $\sprod[\Honez,\Honez^*]{}{}$.
  This gives the following solution components:
  \or
  With $(\ed, \sd)$ in $\setD$ fixed, \BPD2 reduces to
  \begin{align*}
    \smash[b]{\inf_{(u_0, e, s) \in \XD2} \ \obj(e - \ed, s - \sd)}
    \qstq
    e - B_D u_0 - p &= 0 \in \LtwoS, \\
    B_D^* s - f &= 0 \in \Honez^*.
  \end{align*}
  Again, the problem admits a unique global minimizer since
  the constraint map $x \mapsto (B_D^* s, e - B_D u_0)$ is surjective
  and $\obj(e, s)$ is coercive on its kernel,
  $\defset{(u_0, B_D u_0, s)}{(u_0, s) \in \Honez \times N}$.
  The Lagrangian now reads
  \begin{align*}
    L(u_0, e, s, \lambda, \mu)
    &= \obj(e - \ed, s - \sd)
      + \frac1c \sprod[\Honez,\Honez^*]{\lambda}{B_D^* s - f}
      + c \iprod[\LtwoS]{\mu}{e - B_D u_0}
      ,
  \end{align*}
  and the associated KKT conditions become
  \begin{align}
    \label{eq:kktD1}
    \iprod[\LtwoS]{\mu}{B_D \delta u} &= 0 \qfor \delta u \in \Honez, \\
    \label{eq:kktD2}
    \iprod[\LtwoS]{e - \ed + \mu}{\delta e} &= 0 \qfor \delta e \in \LtwoS, \\
    \label{eq:kktD3}
    \iprod[\LtwoS]{s - \sd}{\delta s} +
    \sprod[\Honez,\Honez^*]{\lambda}{B_D^* \delta s}
    &= 0 \qfor \delta s \in \LtwoS, \\
    B_D^* s - f &= 0, \notag \\
    e - B_D u_0 - p &= 0. \notag
  \end{align}
  Condition \eqref{eq:kktD2} is equivalent to $e - \ed + \mu = 0$.
  By \cref{lem:M+N=L2,prop:low-reg-B-B*},
  \eqref{eq:kktD1} implies $\mu \in M^\perp = N$,
  and hence \eqref{eq:kktD1} is equivalent to $B_D^* \mu = 0$.
  Taking $\delta s$ in $\ker B_D^* = N$
  gives $s - \sd \in M$ by \eqref{eq:kktD3}, i.e.,
  $s - \sd = B_D \lambda_0$ with $\lambda_0 \in \Honez$,
  which makes \eqref{eq:kktD3} equivalent to
  $\sprod[\Honez,\Honez^*]{\lambda_0 + \lambda}{B_D^* \delta s} = 0$
  for $\delta s \in \LtwoS$.
  Surjectivity of $B_D^*$ now gives $\lambda = - \lambda_0 \in \Honez$,
  and \eqref{eq:kktD3} is finally equivalent to
  $s - \sd + B_D \lambda = 0$.
  Thus, fixing $(\ed, \sd)$ splits \BPD2 into two problems,
  one for $u$ and $e$ with the multiplier $\mu$ in $N \subset \Hdiv$,
  \begin{align*}
    B_D^* \mu &= 0 \in \LtwoS,
    & e + \mu &= \ed \in \LtwoS,
    & e - B_D u_0 &= p \in \LtwoS,
  \end{align*}
  and another one for $s$ with the multiplier $\lambda$ in $\Honez$,
  \begin{align*}
    s + B_D \lambda &= \sd \in \LtwoS,
    & B_D^* s &= f \in \Honez^*.
  \end{align*}
  Using the orthogonal projections $P_M, P_N$
  and pseudoinverses $B_D^+\: \LtwoS \to \Honez$ (surjective)
  and $(B_D^*)^+\: \Honez^* \to \LtwoS$ (injective with range $M$),
  the unique solutions are readily obtained as
  \fi
  \ifArxiv
  \begin{equation}
    \label{eq:BPD2-x}
    e^* = P_M \ed + P_N p, \qquad
    \mu^* = \ed - e^* = P_N (\ed - p), \qquad
    u_0^* = B_D^+ (e^* - p) = B_D^+ (\ed - p),
  \end{equation}
  respectively
  \begin{equation}
    \label{eq:BPD2-y}
    s^* = P_N \sd + (B_D^*)^+ f, \qquad\qquad
    \lambda^* = B_D^+ (\sd - s^*) = B_D^+ \sd - B_D^+ (B_D^*)^+ f.
    \qedhere
  \end{equation}
  \else
  \begin{align}
    \label{eq:BPD2-x}
    e^* &= P_M \ed + P_N p,
    &\mu^* &= \ed - e^* = P_N (\ed - p),
    &u_0^* = B_D^+ (e^* - p) = B_D^+ (\ed - p),
  \end{align}
  and
  \begin{align}
    \label{eq:BPD2-y}
    s^* &= P_N \sd + (B_D^*)^+ f,
    &\lambda^* &= B_D^+ (\sd - s^*) = B_D^+ \sd - B_D^+ (B_D^*)^+ f,
  \end{align}
  respectively.
  \fi
\end{proof}

\begin{remark}
  Here, $B_D^+ (B_D^*)^+ = \Inv{(B_D^* B_D)}$ holds
  as in \cite[Thm.~1]{Gebhardt_Steinbach:2025}
  since $B_D$ and $B_D^*$ are bounded linear maps
  and hence $P_M, P_N$ can be expressed in terms of each map
  and its pseudoinverse according to \cref{lem:psi}.
\end{remark}

To obtain a unified reformulation of all problems,
we use projected and lifted data
$\lN \dfn P_N p \in N$ and $\lM \dfn (B^*)^+ q + P_M \extN h \in M$,
respectively, with $B^*$ replaced by $B_D^*$ or $B_N^*$ where appropriate.
The restrictions of $\norm[\LtwoS]{}$ to $M$ and $N$
will succinctly be denoted as $\norm[M]{}$ and $\norm[N]{}$.

\begin{theorem}
  \label{thm:BPMR}
  Each of the problems \BPM2, \BPN2 and \BPD2 can
  equivalently be stated as the (unshifted) reduced problem BPMR,
  \begin{equation*}
    \inf_{z \in Z} \
    \frac{c}{2} \norm[N]{P_N \ed - \lN}^2 +
    \frac{1}{2 c} \norm[M]{P_M \sd - \lM}^2
    \qstq z \in \setD,
  \end{equation*}
  in the following sense.
  \begin{enumerate}
  \item For any pair $z_0 = (\ed, \sd_0)$ in $\setD_0$,
    the associated minimizer $x^* = (u_0^*, e^*, s_0^*)$
    of \BPM2, \BPN2 or \BPD2 with $z_0$ fixed
    according to \cref{prop:fixed-z,prop:fixed-zN,prop:fixed-zD}
    gives the same objective value as $z$ for BPMR.
  \item \BPM2, \BPN2 or \BPD2 admits a minimizer $(x^*, z_0^*)$
    if and only if BPMR admits a minimizer $z^*$.
  \end{enumerate}
\end{theorem}

\begin{proof}
  Substituting $e^*$ and $s^* = s_0^* - \extN h$ from
  \eqref{eq:BPM2-x}, \eqref{eq:BPM2-y} or
  \eqref{eq:BPN2-x}, \eqref{eq:BPN2-y} or
  \eqref{eq:BPD2-x}, \eqref{eq:BPD2-y}
  into the objective with $P_N p$ replaced by $\lN$ and
  $(B^*)^+ q + P_M \extN h$ replaced by $\lM$ gives
  \begin{equation*}
    \frac{c}{2} \norm[\LtwoS]{e^* - \ed}^2 +
    \frac{1}{2 c} \norm[\LtwoS]{s_0^* - \sd_0}^2 =
    \frac{c}{2} \norm[N]{\lN - P_N \ed}^2 +
    \frac{1}{2 c} \norm[M]{\lM - P_M \sd}^2.
  \end{equation*}
  This proves claim (i). Claim (ii) is an immediate consequence.
\end{proof}

\begin{remark}
  We have $\lN = 0$ in \BPN2 since $g = 0$,
  and $\lM = (B^*)^+ f$ in \BPD2 since $h = 0$.
  In any case, letting $\ell \dfn \lN + \lM$
  gives a simple reduced form of \BPM0:
  \begin{equation*}
    P_N e + P_M s = \ell, \quad \cg(e, s) = 0.
  \end{equation*}
  Moreover, the following observation implies
  $\extN h \in M$ and hence $\lM = (B^*)^+ q + \extN h$.
\end{remark}

\begin{lemma}
  \label{lem:orth-Dperp}
  $D^\perp$ is orthogonal to $N$ both in $\Hdiv$ and in $\LtwoS$.
\end{lemma}

\begin{proof}
  In $\Hdiv = D \oplus D^\perp$ this holds since $N \subset D$.
  For any $(s, t) \in N \times D^\perp$ it follows from $\div s = 0$
  that $0 = \iprod[\Hdiv]{s}{t} = \iprod[\LtwoS]{s}{t} + 0$.
\end{proof}

To obtain a similar result for $G^\perp$ and $G \cap R$,
consider on $\Hone$ the geometry induced by the graph norm of $\Grad$,
giving the space $\Hsym$, which coincides with $\Hone$ as set,
but is equipped with the scalar product
\begin{equation*}
  \iprod[\Hsym]{u}{v} \dfn
  \iprod[\LtwoR]{u}{v} + \iprod[\LtwoS]{\Grad u}{\Grad v},
\end{equation*}
similarly for $\Honez$ and $G = \HoneD$.
The norms on these spaces are equivalent
by Remark \labelcref{rem:closed-op},
and all prior results for $\Hone$, $\Honez$, $\HoneD$
remain valid for $\Hsym$, $\Hsym[0]$, $\Hsym[D]$.
In \cref{lem:ext-N}, we adapt the notation to clarify the key difference:
$\Hsym = G \oplus G_s^\perp$ with $\extD_s\: \Hphalf \to G_s^\perp$.

\begin{lemma}
  \label{lem:orth-Gperp}
  $G_s^\perp$ is orthogonal to $G \cap R$ both in $\Hsym$ and in $\LtwoR$.
\end{lemma}

\begin{proof}
  In $\Hsym$ this is obvious.
  For any $(u, v) \in (G \cap R) \times G_s^\perp$ it follows from $\Grad u = 0$
  that $0 = \iprod[\Hsym]{u}{v} = \iprod[\LtwoR]{u}{v} + 0$.
\end{proof}

This result suggests $\Hsym$ as the most appropriate space
for the displacement fields:
it is more natural than $\Hone$ in terms of its geometry.

\begin{remark}
  If any \ifArxiv of the \fi data-driven problem\ifArxiv s\fi\
  admits a minimizer,
  then the unique physical \ifArxiv strain-stress \fi fields $(e^*, s^*) \in Z$
  determine all associated auxiliary fields $(\ed^*, \sd^*) \in \setD$
  via a set-valued \emph{feedback operator} $\Phi$
  given by \emph{pointwise minimization} over $\Omega$,
  \begin{align*}
    \Phi\: Z &\rightrightarrows \setD,
    &\Phi(e^*, s^*)(x) &\dfn \argmin_{(\mathfrak e, \mathfrak s) \in \data} \
    \obj(e^*(x) - \mathfrak e, s^*(x) - \mathfrak s) \qfor x \in \Omega.
  \end{align*}
  This suggests alternating direction methods (ADM)
  as suitable solution algorithms.
\end{remark}

\section{Spatial Dimension One}
\label{sec:1d}

To add concreteness to the theory above, we now give
explicit representations of the relevant spaces, operators, etc.\
for $\Omega = (0, L)$ with dimension $d = 1$.
Since $\Rd = \Md = \Sd = \R$, $\Gamma = \set{0, L}$,
and $\div = \grad = \Grad = \pd_x$ (weak derivative),
everything simplifies considerably.
The basic Hilbert spaces are
\begin{align*}
  \LtwoR &= \Ltwo = \LtwoS, \\
  \Hsym = \Hone &= H^1(\Omega) = \Hdiv, \\
  \Hsym[0] = \Honez &= H_0^1(\Omega) = \Hdivz.
\end{align*}
We have either $\bdD = \Gamma$ (Dirichlet), $\bdN = \Gamma$ (Neumann),
or $\bdD = \set{0}$ and $\bdN = \set{L}$ or vice versa (mixed).
The regularity of boundary data is not an issue since
\begin{equation*}
  \Hphalf \cong \R^2 \cong \Hnhalf = \Hnhalff.
\end{equation*}
Let $\uvec \dfn L^{-1/2} \chi_\Omega \in \Ltwo$,
with $\norm[\Ltwo]{\uvec} = 1$.
Then, we have the ``rigid body spaces''
(cf.\ \cite[Prop.~6]{Gebhardt_Lange_Steinbach:2025}
and \cite[Prop.~3]{Gebhardt_Steinbach:2025})
\begin{gather*}
  R = \defset{(x \mapsto b)}{b \in \R} =
  \Span\set{\uvec} = \set{\text{constant $L^2$ functions}}, \\
  Q = \defset{u \in \Ltwo}{\textstyle\int u = 0} =
  \set{\uvec}^\perp = \set{\text{zero-mean $L^2$ functions}}.
\end{gather*}
With a sloppy notation for $G, D \subseteq H^1(\Omega)$,
we obtain further relevant subspaces:
\begin{align*}
  \bdD &= \Gamma\:
  & G &= H_0^1(\Omega), & D &= H^1(\Omega),
  & M &= Q, & N &= R, \\
  \bdD &= \set{0}\:
  & G &= \defset{u}{u(0) = 0}, & D &= \defset{s}{s(L) = 0},
  & M &= \Ltwo, & N &= \set{0}, \\[-2pt]
  \bdD &= \set{L}\:
  & G &= \defset{u}{u(L) = 0}, & D &= \defset{s}{s(0)  = 0},
  & M &= \Ltwo, & N &= \set{0}, \\
  \bdD &= \0\:
  & G &= H^1(\Omega), & D &= H_0^1(\Omega),
  & M &= \Ltwo, & N &= \set{0}.
\end{align*}
The extension and lifting operators associated with $G$ and $D$ read
(cf.\ \cite[Prop.~1]{Gebhardt_Steinbach:2025}):
\begin{align*}
  \bdD &= \Gamma\:
  &(\extD g)(x)
  &= \frac{\sinh(L - x)}{\sinh(L)} g(0)
    + \frac{\sinh(x)}{\sinh(L)} g(L),
    \qquad (\extN h)(x) = 0, \\
  \bdD &= \set{0}\:
  &(\extD g)(x)
  &= \frac{\cosh(L - x)}{\cosh(L)} g(0),
    \qquad\ \ (\extN h)(x) = +\frac{\cosh(x)}{\cosh(L)} h(L), \\[-1pt]
  \bdD &= \set{L}\:
  &(\extD g)(x)
  &= \frac{\cosh(x)}{\cosh(L)} g(L),
    \qquad\ \ (\extN h)(x) = -\frac{\cosh(L - x)}{\cosh(L)} h(0), \\
  \bdD &= \0\:
  &(\extD g)(x)
  &= 0,
    \qquad(\extN h)(x) =
    \frac{\sinh(x)}{\sinh(L)} h(L)
    - \frac{\sinh(L - x)}{\sinh(L)} h(0).
\end{align*}
The inhomogeneities in \BPM1 and \BPM2 are
$p = \pd_x \extD g$ and $q = f + \pd_x \extN h$,
yielding $p = 0$ in \BPN2 and $q = f$ in \BPD2 as above.
In \BPN2, we have $G_Q = H^1(\Omega) \cap Q$,
the zero-mean $H^1$ functions.
The projections $P_M$ and $P_N$ are trivial if $\bdD \ne \Gamma$;
they have the following explicit representations
in the pure Dirichlet case, $\bdD = \Gamma$:
\begin{align*}
  P_N s &= \iprod[\Ltwo]{\uvec}{s} \uvec \in \ker B^*,
  & P_M s &= s - P_N s \in \im B.
\end{align*}
The pseudoinverse $B^+$ in \cref{prop:fixed-z}
acts as an integral operator on $p \in \Ltwo$;
cf.\ \cite[Prop.~6]{Gebhardt_Lange_Steinbach:2025}
and \cite[Prop.~3]{Gebhardt_Steinbach:2025}:
\begin{align*}
  \bdD &= \Gamma\:
  &(B^+ p)(x) &= \int_0^x p(y) \dy - \frac xL \int_0^L p(y) \dy, \\
  a \in \set{0, L}, \quad\bdD &= \set{a}\:
  &(B^+ p)(x) &= \int_a^x p(y) \dy, \\
  \bdD &= \0\:
  &(B^+ p)(x) &= \int_L^x p(y) \dy + \frac1L \int_0^L y p(y) \dy.
\end{align*}
Likewise, $(B^*)^+$ acts as an integral operator on $q \in \Ltwo$;
cf.\ \cite[Prop.~3]{Gebhardt_Steinbach:2025}:
\begin{align*}
  \bdD &= \Gamma\:
  &((B^*)^+ q)(x) &= \int_x^L q(y) \dy - \frac1L \int_0^L y q(y) \dy, \\
  a \in \set{0, L}, \quad\bdD &= \set{a}\:
  &((B^*)^+ q)(x) &= \int_x^{L - a} q(y) \dy, \\
  \bdD &= \0\:
  &((B^*)^+ q)(x) &= \int_x^0 q(y) \dy + \frac xL \int_0^L q(y) \dy.
\end{align*}
The expressions for $B^+$ and $(B^*)^+$ in the case $\bdD = \0$
hold also for $B_N^+$ and $(B_N^*)^+$ in \cref{prop:fixed-zN},
and in the case $\bdD = \Gamma$
for $B_D^+$ and $(B_D^*)^+$ in \cref{prop:fixed-zD}.
(Of course, for $(B_D^*)^+$ this pertains only to
the dense subspace $\Ltwo \subset H_0^1(\Omega)^*$.)

Finally, the data-driven problems
with pure Neumann or mixed boundary data
are \emph{universally solvable}, as follows.
For pure Neumann data,
the compatibility condition $q \in Q$
becomes $\int_0^L f(y) \dy = h(0) - h(L)$,
and \BPN2 (without the shift $\extN h$) reads
\begin{equation*}
  \inf \ \obj(e - \ed, s - \sd) \qstq
  \pd_x u = e, \quad
  \pd_x s = -f, \quad
  s(0) = h(0), \quad s(L) = h(L).
\end{equation*}
From the boundary conditions, we have the unique optimal stress field
\begin{equation*}
  s^*(x) =
  h(0) - \int_0^x f(y) \dy =
  h(L) + \int_x^L f(y) \dy.
\end{equation*}
Next, we obtain (possibly non-unique) optimal auxiliary fields $(\ed^*, \sd^*)$
by minimizing $\abs{s^*(x) - \sd(x)}$ \emph{pointwise} over $x \in \Omega$.
Finally, we set $e^* \dfn \ed^*$ and obtain $u^*$ in $Q$ as
\begin{equation*}
  u^*(x) \dfn \~u(x) - \frac1L \int_0^L \~u(y) \dy
  \qtextq{where}
  \~u(x) \dfn \int_0^x e^*(y) \dy.
\end{equation*}
The construction for \BPM2 with mixed boundary data is similar;
cf.\ \cite[Thm.~7]{Gebhardt_Steinbach:2025}:
\begin{equation*}
  \inf \ \obj(e - \ed, s - \sd) \qstq
  \pd_x u = e, \quad
  \pd_x s = -f, \quad
  u(0) = g(0), \quad s(L) = h(L).
\end{equation*}
Again, we have the unique optimal stress field
$s^*(x) = h(L) + \int_x^L f(y) \dy$
from which we obtain $(\ed^*, \sd^*)$ in $\setD$, then $e^* \dfn \ed^*$,
and finally the unique displacement field
\begin{equation*}
  u^*(x) = g(0) + \int_0^x e^*(y) \dy.
\end{equation*}
Note that closedness of the data set $\data$ (possibly given by $\cg = 0$)
suffices for existence of these data-driven solutions.
Note also that they solve \BPM0 if and only if $\sd^* = s^*$.

\section{Summary}
\label{sec:summary}

We have presented a complete structural analysis
of the classical static elasticity problem at small deformations
and of its data-driven counterpart in Hilbert space.
The key spaces, operators and fields with their interrelations
and with all four orthogonal decompositions
are illustrated in \cref{fig:cd}.

\begin{figure}
  \centering
  \let\emblue\emred\input{cd-bpm.tex}
  \caption{Commutative diagram of key spaces, operators and fields.
    Objects in red depend on the boundary partitioning;
    a red $\oplus$ connects spaces that are orthogonal
    in $L^2$ and in $H\tsp{sym}$ or $H\tsp{div}$.}
  \label{fig:cd}
\end{figure}
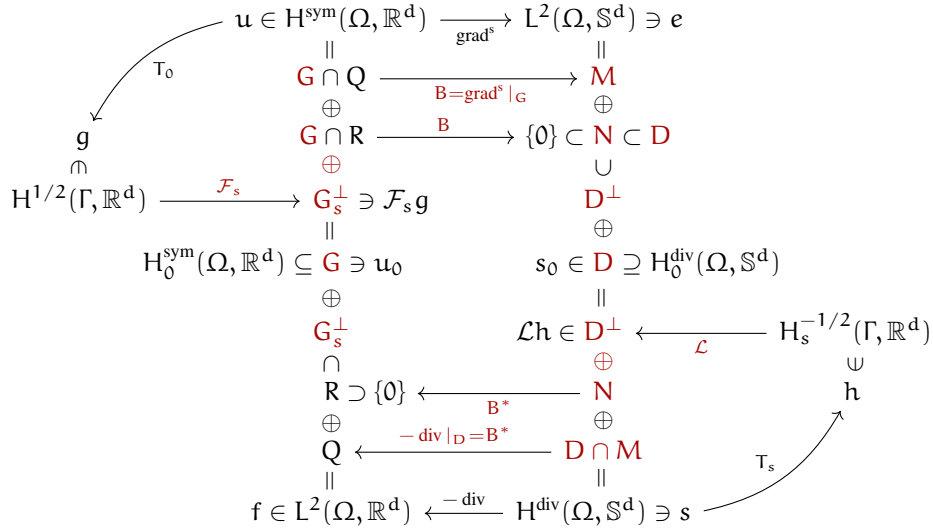

Our theory generalizes earlier work for one spatial dimension
in discretized form \cite{Gebhardt_Lange_Steinbach:2025}
and in Hilbert space \cite{Gebhardt_Steinbach:2025},
where only the pure Dirichlet problem is nontrivial
and adjoint operators appear naturally.
This adjoint duality extends to arbitrary dimension
(we found this only in \Cite{Kurula_Zwart:2012}
and afterwards in \Cite{Gudoshnikow_Krizk:2025})
where it is embedded in a much richer, primarily geometric structure
that boils down to orthogonal decompositions of Hilbert spaces
induced by the duality, by infinitesimal rigid body motion,
and by the extension and lift maps for boundary data.

Although the work in \Cite{Gudoshnikow_Krizk:2025}
is similar in spirit to ours with a focus on geometric structure
and with unbounded adjoint operators as a key concept
(inspired by \Cite{Kurula_Zwart:2012}),
it differs in several respects.
Firstly, it is more special in that
a specific material law is chosen (inhomogeneous linear elasticity),
pure Neumann problems are not addessed,
and Neumann boundary data are taken in $L^2(\Gamma, \Rd)$, not $\Hnhalff$.
More importantly, several other aspects (often seemingly innocuous)
direct the focus to secondary technical issues
so that symmetries and structural properties are overlooked:
stress $s$ is taken as the ``main unknown'',
the equilibrium is stated in variational form,
with Dirichlet and Neumann boundary data appearing asymmetrically
and being treated asymmetrically later on,
and the key unbounded adjoint operators involve Neumann boundary data
via the dual Lions--Magenes space.
In contrast, we keep symmetries and uncover structural properties
by relying on the dual trace spaces with their duality pairing
(inspired by \Cites{Conti2018,Conti2020}),
by employing pseudoinverses where possible,
by treating all physical states equally (in particular strains and stresses),
and by keeping PDEs and boundary conditions cleanly separated
(via extension and lift maps).

From a mechanical viewpoint,
the deepest structure revealed by the orthogonal decompositions
is the splitting between kinematics and equilibrium:
$\LtwoS = M \oplus N$.
The first component is generated by displacement fields,
whereas the second consists of self-equilibrated stresses.
Once this decomposition is established,
the remaining ingredients of the formulation
are determined by the boundary conditions and source terms.
The secondary structure identified provides the trace operators
together with their associated extension and lifting maps.
This structure allows the boundary conditions to be absorbed into shifted fields
and transforms the original problem into an optimization problem
over linear spaces with homogeneous boundary conditions and,
where necessary, with rigid body motions factored out.
In this sense, the presented framework identifies
the intrinsic degrees of freedom of the elasticity problem
as displacement fields and self-equilibrated stress fields,
while boundary conditions and external loads
act only through appropriate shifts of these fundamental spaces.
This also applies to data-driven continuum mechanics.

We believe that the presented theory provides a firm basis
for analyzing existence and uniqueness of solutions
to data-driven elasticity problems (our original motivation),
and we hope that our results will be useful in general
for mathematical elasticity and for the related areas
mentioned in the introduction.

\bibliographystyle{siam}
\bibliography{ddcm}

\typeout{get arXiv to do 4 passes: Label(s) may have changed. Rerun}
\end{document}

%% file: cd-bpm.tex
\begin{center}
  \newcommand\eql[1][=]{\ar[d, phantom, "\rotatebox{90}{$#1$}"]}
  \newcommand\opl[1][]{\ar[d, phantom, "#1\oplus"]}
  \newcommand\red{\textcolor{LUH-red}}
  \begin{tikzcd}[row sep=5pt]
    &
    u \in \Hsym \eql \ar[ldd, bend right=24, "T_0"]
    \ar[r, "\Grad"'] &
    \LtwoS \ni e \eql \\
    &
    \red G \cap Q \opl \ar[r, "\;\red{B = \Grad|_G}"'] &
    \red M \opl \\
    \; g \eql[\ni] &
    \red G \cap R \opl[\red] \ar[r, "\red B\,"] &
    \set{0} \subset \red N
    \rlap{${} \subset \red D$} 
    \hphantom{{} \subset \set{0}}
    \eql[\subset] \\
    \Hphalf \ar[r, "\red{\extD_s}"] &
    \red{G_s^\perp} \rlap{${}\ni \extD_s g$} \eql &
    \red{D^\perp} \opl \\
    &
    \llap{$\Hsym[0] \subseteq {}$} \red G \rlap{${}\ni u_0$} \opl &
    \llap{$s_0 \in{}$} \red D \rlap{${} \supseteq \Hdivz$} \eql \\
    &
    \red{G_s^\perp} \eql[\supset] &
    \llap{$\extN h \in{}$} \red{D^\perp} \opl[\red] &
    \Hnhalf \ar[l, "\red\extN"] \eql[\in] \\
    &
    \hphantom{\set{0} \supset {}} R \supset \set{0} \opl &
    \red N \opl \ar[l, "\red{B^*}"] & h \; \\
    &
    Q \eql &
    \red{D \cap M} \eql \ar[l, "\,\red{-\div|_D = B^*}"'] \\
    &
    f \in \LtwoR &
    \Hdiv \ni s \ar[l, "-\div"'] \ar[ruu, bend right=24, "T_s"]
  \end{tikzcd}
\end{center}